\documentclass[10pt, oneside]{article}

\usepackage[dvips]{graphicx}
\usepackage{pstricks}
\usepackage{pst-plot}

\usepackage{amssymb}
\usepackage{amsmath}
\usepackage{amsthm}
\usepackage{color}

\newtheorem{theorem}{Theorem}[section]

\newtheorem{lemma}[theorem]{Lemma}

\theoremstyle{definition}
\newtheorem{definition}{Definition}[section]

\theoremstyle{remark}
\newtheorem{remark}{Remark}[section]

\newcommand{\RR}{\mathbb{R}}
\newcommand{\NN}{\mathbb{N}}

\begin{document}
\title{Bounded solutions of a $k$-Hessian equation in a ball \thanks{Partially supported by Grant FONDECYT 1150230}}

\author{Justino S\'anchez\,\,$^{a}$\;\; and
Vicente Vergara \,$^{a,b}$}
\date{}
\maketitle

\begin{center}
$^a$\,\footnotesize{Departamento de Matem\'{a}ticas, Universidad de La Serena\\
 Avenida Cisternas 1200, La Serena 1700000, Chile.}
\\email: jsanchez@userena.cl
\end{center}
\begin{center}
$^b$\,\footnotesize{Universidad de Tarapac\'{a}, Avenida General Vel\'{a}squez 1775,\\
Arica, Chile.}
\\email: vvergaraa@uta.cl
\end{center}

\begin{abstract}
We consider the problem
\begin{equation}\label{Eq:Abstract}
\begin{cases}
S_k(D^2u)= \lambda (1-u)^q &\mbox{in }\;\; B,\\
u <0 & \mbox{in }\;\; B,\\
u=0 &\mbox{on }\partial B,
\end{cases}
\end{equation}
where $B$ denotes the unit ball in $\RR^n$, $n>2k$ ($k\in \NN$), $\lambda>0$ and $q > k$. We study the existence of negative bounded radially symmetric solutions of \eqref{Eq:Abstract}. In the critical case, that is when $q$ equals Tso's critical exponent $q=\frac{(n+2)k}{n-2k}=:q^*(k)$, we obtain exactly either one or two solutions depending on the parameters. Further, we express such solutions explicitly in terms of Bliss functions. The supercritical case is analysed following the ideas develop by Joseph and Lundgren in their classical work \cite{JoLu73}. In particular, we establish an Emden-Fowler transformation which seems to be new in the context of the $k$-Hessian operator. We also find a critical exponent, defined by
\begin{equation*}
q_{JL}(k)=
\begin{cases}
k\frac{(k+1)n-2(k-1)-2\sqrt{2[(k+1)n-2k]}}{(k+1)n-2k(k+3)-2\sqrt{2[(k+1)n-2k]}}, & n>2k+8,\\
\infty, & 2k < n \leq 2k+8,
\end{cases}
\end{equation*}
which allows us to determinate the multiplicity of the solutions to \eqref{Eq:Abstract} int the two cases $q^*(k)\leq q < q_{JL}(k)$ and $q\geq q_{JL}(k)$. Moreover, we point out that, for $k=1$, the exponent $q_{JL}(k)$ coincides with the classical Joseph-Lundgren exponent.
\end{abstract}


\section{Introduction and main results}
Let $k \in \NN$ and let $\Omega$ be a suitable bounded domain in $\RR^n$. We consider the nonlinear problem
\begin{equation}\label{Eq:f}
\begin{cases}
S_k(D^2u)= f(u)&\mbox{in }\;\; \Omega,\\
u <0 & \mbox{in }\;\; \Omega,\\
u=0 &\mbox{on }\;\; \partial\Omega,
\end{cases}
\end{equation}
where $S_k(D^2u)$ stands for the $k$-Hessian operator of $u$ and $f$ is a given nonlinear source. Problem \eqref{Eq:f} has been studied extensively by many authors in different settings. See e.g. \cite{CaNS85, ChWa01, Trudinger95, Trudinger97, Tso89, Urbas90, Wang94}.\\
The $k$-Hessian operator $S_k$ is defined as follows. Let $u\in C^2(\Omega)$, $1\leq k\leq n$, and let $\Lambda=(\lambda_1,\lambda_2,...,\lambda_n)$ be the eigenvalues of the Hessian matrix $(D^2u)$. Then the $k$-Hessian operator is given by
\[
S_k(D^2u)=P_k(\Lambda)=\sum_{1\leq i_1<...<i_k\leq n}\lambda_{i_1}...\lambda_{i_k},
\]
where $P_k(\Lambda)$ is the $k$-th elementary symmetric polynomial in the eigenvalues $\Lambda$, see e.g. \cite{Wang94, Wang09}. Note that $\{S_k:k=1,...,n\}$ is a family of operators which contains the Laplace operator ($k=1$) and the Monge-Amp\`{e}re operator ($k=n$). The monograph \cite{CaMi98} is devoted to applications of Monge-Amp\`{e}re equations to geometry and optimization theory. This family of operators has been studied extensively, see e.g. \cite{Jacobsen99, Tso90} and the references therein. Recently, this class of operators has attracted renewed interest, see e.g. \cite{Bran13, Gavitone09, Gavitone10, NaTa15, DeGa14, WaBa13, WaXu14}.\\
We point out that the $k$-Hessian operators are fully nonlinear for $k\neq 1$. Further, they are not elliptic in general, unless they are restricted to the class
\begin{equation}\label{phi-k-0}
\Phi_0^k(\Omega)=\{u\in C^2(\Omega)\cap C(\overline{\Omega}): S_{i}(D^2 u)\geq 0\;\;\mbox{in}\;\;\Omega,\, i=1,...,k,\,u=0\;\; \text{on}\;\; \partial\Omega \}.
\end{equation}
Observe that $\Phi_0^k(\Omega)$ belongs to the class of subharmonic functions. Further, the functions in $\Phi_0^k(\Omega)$ are negative in $\Omega$ by the maximum principle, see \cite{Wang94}. The $k$-Hessian operator defined on $\Phi_0^k(\Omega)$ imposes certain geometry restrictions on $\Omega$. More precisely, domains called admissible are those whose boundary $\partial\Omega$ satisfies the inequality
\begin{equation}\label{curv:1}
P_{k-1}(\kappa_1, ..., \kappa_{n-1})\geq 0,
\end{equation}
where $\kappa_1, ..., \kappa_{n-1}$ denote the principal curvatures of $\partial\Omega$ relative to the interior normal. A typical example of a domain $\Omega$ for which \eqref{curv:1} holds is a ball. For more details we refer the interested reader to \cite{Wang09}.

\begin{remark}
Problem \eqref{Eq:f} can be easily reformulated in order to study positive solutions under the change of variable $v=-u$, which in turn yields $S_k(D^2u)=(-1)^kS_k(D^2v)$ by the $k$-homogeneity of the $k$-Hessian operator.
\end{remark}
\medbreak

Now observe that, if $u\in \Phi_0^k(\Omega)$, then the right hand side of (\ref{Eq:f}) must be nonnegative. Typical examples of nonlinear terms $f$ appearing in the literature are $f(u) = |\lambda u|^p$ (see \cite{Tso90}), $f(u)= \lambda e^{-u}$ (see \cite{Shandrasekhar85, Fowler31, Gelfand63, JoLu73} for $k=1$ and \cite{Jacobsen99, Jacobsen04, JaSc02, JaSc04} for $1\leq k \leq n$) and  $f(u)=\lambda(1+u)^p$ (see \cite{BrVa97, JoLu73} for $k=1$.)
The seminal contribution on the analysis of critical values for $k=1$ with a polynomial and exponential source was made by Joseph and Lundgren in \cite{JoLu73}. In general, for problems of Gelfand type for $k\geq 1$, the first result ($k=n$) in the radial case is due to Cl\'ement et al. \cite{ClFM96} and for $1\leq k \leq n$ to Jacobsen \cite{Jacobsen04}.

\medbreak

Next, for our purposes we give some general notions of solutions to (\ref{Eq:f}). As usual, a classical solution (or solution) of (\ref{Eq:f}) is a function $u\in \Phi_0^k(\Omega)$ satisfying the equation in (\ref{Eq:f}). We recall the version of the method of super and subsolutions for \eqref{Eq:f}, see \cite[Theorem 3.3]{Wang94} for more details.
\begin{definition}
A function $u\in\Phi^k(\Omega):=\{u\in C^2(\Omega)\cap C(\overline{\Omega}): S_{i}(D^2 u)\geq 0\;\;\mbox{in}\;\;\Omega,\, i=1,...,k,\}$ is called a {\it subsolution} (resp. {\it supersolution}) of \eqref{Eq:f} if
\begin{equation*}
\begin{cases}
S_k(D^2u)\geq (\mbox{resp.}\leq)& f(u)\;\;\mbox{in }\;\; \Omega,\\
u\leq (\mbox{resp.}\geq)\;\; 0&\qquad\;\,\mbox{on }\; \partial \Omega.
\end{cases}
\end{equation*}
\end{definition}
Note that the trivial function $u\equiv 0$ is always a supersolution.

The following concept is needed to establish a general result on the existence of solutions to problem \eqref{Eq:f}.
\begin{definition}
We say that a function $v$ is a {\it maximal} solution of \eqref{Eq:f} if $v$ is a solution of \eqref{Eq:f} and, for each subsolution $u$ of \eqref{Eq:f}, we have $u\leq v$.
\end{definition}
We note that the notion of maximal solution of \eqref{Eq:f} seems to be new in the context of $k$-Hessian equations. In case the $k=1$, this notion corresponds to the usual minimal (positive) solution, see e.g. the monograph \cite{DuPa11} and the references therein.

In this article we study problem (\ref{Eq:f}) on the unit ball $B$ of $\RR^n$ with a polynomial source, i.e, the problem
\begin{equation}\label{Eq:f:pol}
\begin{cases}
S_k(D^2u)= \lambda (1-u)^q &\mbox{in }\;\; B,\\
u <0 & \mbox{in }\;\; B,\\
u=0 &\mbox{on }\partial B,
\end{cases}
\end{equation}
where $\lambda\in\RR$ is a parameter and $q>0$. We recall that the $k$-Hessian operator in radial coordinates can be written as $S_k(D^2u)=c_{n,k}r^{1-n}\left(r^{n-k}(u')^k \right)'$, where $c_{n,k}$ is defined by $c_{n,k}=\binom{n}{k}/n$ and $\binom{n}{k}$ denotes the binomial coefficient. Next, in order to state our main result, we write \eqref{Eq:f:pol} in radial coordinates, i.e.,
\begin{equation*}
(P_{\lambda})\qquad
\begin{cases}
c_{n,k}r^{1-n}\left(r^{n-k}(u')^k \right)' = \lambda (1-u)^q\,,\quad 0<r<1,\\
u  < 0 \,, \hspace{4.65cm} 0<r<1,\\
u'(0)  =0,\, u(1)=0. &
\end{cases}
\end{equation*}
Now we introduce the space of functions $\Phi_0^k$ defined on $\Omega=(0,1)$ as in \eqref{phi-k-0}, for problem $(P_{\lambda})$:
\[
\Phi_0^k=\{u\in C^2((0,1))\cap C^1([0,1]): \left(r^{n-i}(u')^i \right)'\geq 0\;\;\mbox{in}\;\;(0,1) ,\, i=1,...,k,\,u'(0)=u(1)=0\}.
\]
We note that the functions in $\Phi_0^k$ are non positive on $[0,1]$. However, if $\left(r^{n-i}(u')^i \right)'> 0$ for all $i=1,\ldots, k$, then any function in $\Phi_0^k$ is negative and strictly increasing on $(0,1)$. This in turn implies that, if we are looking for solutions of ($P_\lambda$) in $\Phi_0^k$, then the parameter $\lambda$ must be positive.

\begin{definition}
Let $\lambda>0$. We say that a function $u \in C([0,1])$ is:
\begin{itemize}
\item[(i)] a {\it classical solution} of ($P_\lambda$) if $u\in \Phi_0^k$ and the equation in ($P_\lambda$) holds;
\item[(ii)] an {\it integral solution} of ($P_\lambda$) if $u$ is absolutely continuous on $(0,1]$, $u(1)=0$, $\int_0^1 r^{n-k}(u'(r))^{k+1} dr <\infty$ and the equality
\[
c_{n,k}r^{n-k}(u'(r))^k = \lambda \int_0^r s^{n-1}(1-u(s))^qds,\, \, \text{ a.a. } r\in (0,1),
\]
holds whenever the integral exists.
\end{itemize}
\end{definition}

The concept of integral solution was introduced in \cite{ClFM96} for a more general class of radial operators, see e.g. \cite{ClFM96} and the references therein.  The standard concept of weak solution is equivalent in this case to the notion of integral solution, see \cite[Proposition 2.1]{ClFM96}.

\medbreak

The main goal of this paper is to describe the set of negative bounded radially symmetric solutions to \eqref{Eq:f:pol} in terms of the parameters. Our statements contain some classical results (i.e. $k=1$), see \cite{JoLu73}. We compute a critical exponent of the Joseph-Lundgren type, defined by
\begin{equation}\label{Exp:critical:Intro}
q_{JL}(k)=
\begin{cases}
k\frac{(k+1)n-2(k-1)-2\sqrt{2[(k+1)n-2k]}}{(k+1)n-2k(k+3)-2\sqrt{2[(k+1)n-2k]}}, & n>2k+8,\\
\infty, & 2k < n \leq 2k+8.
\end{cases}
\end{equation}
The Joseph-Lundgren exponent, i.e.,
\[
q_{JL}(1)=\frac{n-2\sqrt{n-1}}{n-4-2\sqrt{n-1}}
\]
was introduced in \cite{JoLu73}. We prove that $q_{JL}(k)$ plays the same role as the Joseph-Lundgren exponent. Another important exponent appearing in the analysis of the boundedness of solutions to \eqref{Eq:f:pol} is given by
\begin{equation}\label{q:ast:k}
q^*(k)=\frac{(n+2)k}{n-2k},
\end{equation}
which is smaller than $q_{JL}(k)$. The value $q^*(k)$ is well-known as the critical exponent in the study of the quasilinear $k$-Hessian operator, see \cite{Tso90} for more details. As soon as this critical exponent is crossed, a drastic change in the number of solutions of \eqref{Eq:f:pol} occurs.

Assuming that \eqref{Eq:f:pol} admits a solution $u_{\lambda}$ for some $\lambda$, we may define the positive constant
\begin{equation}\label{lambda-ast}
\lambda^\ast=\sup\{\lambda>0:\, \text{ there is a solution } u_{\lambda}\in C^2(B) \text{ of } \eqref{Eq:f:pol}\}.
\end{equation}

We show that $\lambda^*$ is finite in Theorem \ref{Main:3:Intro} below.
Now we state our main result.
\begin{theorem}\label{Main:1:Thm1} Let $q>k$ and $n>2k$. Let $q^*(k)$ and $q_{JL}(k)$ be as in \eqref{Exp:critical:Intro} and \eqref{q:ast:k}, respectively.
\begin{itemize}
	\item[(I)] If $q^\ast(k) < q < q_{JL}(k)$ and $\lambda$ is close to but not equal to
	\begin{equation}\label{Lambda:Tilde:Intro}
\tilde{\lambda}(k):= \tau^k(n-2k-k\tau),	
	\end{equation}
	where $\tau=\frac{2k}{q-k}$, then $(P_\lambda)$ has a large (finite) number of solutions. In addition, if $\lambda=\tilde{\lambda}(k)$ then there exists infinitely many solutions of $(P_\lambda)$.
	\item[(II)] If $n>2k+8$, $q\geq q_{JL}(k)$ and $\lambda\in (0, \lambda^*)$, then there exists only one solution of $(P_\lambda)$. Moreover, $\lambda^*=\tilde{\lambda}(k)$.
\end{itemize}
\end{theorem}
For the reader's convenience, we sketch the proof of Theorem \ref{Main:1:Thm1}. To this end, we briefly discuss, for our case, the method  introduced in \cite{JoLu73}. We first make a rescaling of problem ($P_{\lambda}$) to obtain
\begin{equation}\label{Eq:v:Intro}
\begin{cases}
\left(s^{n-k}(v')^k\right)'= \tilde{\lambda}(k) s^{n-1}(-v)^q, & s>0,\\
v(0) = -1,\\
v'(0) =0.
\end{cases}
\end{equation}
Note that $\tilde{\lambda}(k)>0$ if and only if $q > \frac{nk}{n-2k}$. Next, we introduce a dynamical system associated to \eqref{Eq:v:Intro}, through the following change of variables of Emden-Fowler type:
\[
y(t) = \left(\frac{dv}{dt}\right)^k e^{\frac{2k^2}{q-k}t},\, z(t) = \tilde{\lambda}(k) e^{\frac{2kq}{q-k}t} (-v(e^t))^q,\, s=e^t.
\]
That is,
\begin{equation}\label{Eq:Dyn:Intro}
\begin{cases}
\frac{dy}{dt}= z(t) - \frac{a}{q-k} y(t),\\
\,\\
\frac{dz}{dt} = \frac{2kq}{q-k} z(t) - q\tilde{\lambda}(k)^{\frac{1}{q}}y(t)^{\frac{1}{k}}z(t)^{1-\frac{1}{q}},\\
\end{cases}
\end{equation}
where
\begin{equation}\label{def:a:Intro}
a:=q(n-2k) - nk.
\end{equation}
The linearization of \eqref{Eq:Dyn:Intro} at the critical points allows for the analysis of the dynamical system on the phase plane, which depends on the eigenvalues of the Jacobian matrix. The location of these eigenvalues on the complex plane is determined by the sign of $2k-a$. The case $2k-a>0$ corresponds to instability, $2k-a<0$ to stability, and for $2k-a=0$ to a center. We point out that $2k-a=0$ is equivalent to defining $q=q^\ast(k)$.

We first note that $2k-a<0$ is equivalent to $q>q^*(k)$ and by \cite{ClMM98} there exists a classical global solution of \eqref{Eq:v:Intro}. In case $2k-a>0$ (i.e. $q<q^*(k)$) we cannot claim the existence of a global solution of \eqref{Eq:v:Intro} and thus discussion of this case is excluded. In the stability case, that is $q>q^*(k)$, depending on the behavior of the discriminant $\varDelta$ of the Jacobian matrix, which is given by
\[
\varDelta(a)=(q-k)^{-2}[(2k-a)^2-8a(q-k)],
\]
we obtain two types of orbits (see figure 1 below). To compute the exponent $q_{JL}(k)$, we first solve the equation $\varDelta(a)=0$ and then choose the larger root of it. Replacing this root into the definition of $a$ in \eqref{def:a:Intro}, we obtain
\begin{equation}\label{f:Intro}
n-2k=f_k(q),
\end{equation}
where
\[
f_k(q)=\frac{4q}{q-k}+4\sqrt{\frac{q}{q-k}}+\frac{2k}{q-k}(k-1).
\]
Observe that, for $k=1$, $f_1(q)$ coincides with the function $f$ introduced in \cite{JoLu73}. Now we point out that the larger root of \eqref{f:Intro} is exactly $q_{JL}(k)$ defined in \eqref{Exp:critical:Intro} in the case $n>2k+8$.

\subsection*{}
\vspace{-1cm}
\begin{figure}[!ht]
\begin{center}
\includegraphics{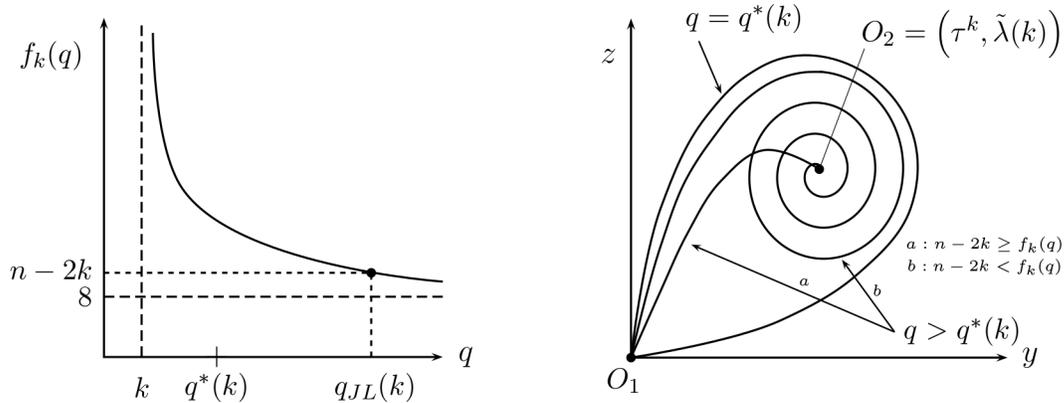}
\vspace{-0.4cm}
\caption{Phase plane analysis for $y$-$z$.}
\end{center}
\end{figure}
In figure 1, the graph of $f_k(q)$ shows that we have two possibilities: ($a$) $n-2k\geq f_k(q)$ if $q\geq q_{JL}(k)$ or ($b$) $n-2k< f_k(q)$ if $q<q_{JL}(k)$. Further, we see that $f_k(q)$ decays asymptotically to the horizontal line $8$, which ensures the existence of a unique positive number $q_{JL}(k)$ whenever $n-2k>8$, by \eqref{f:Intro}. On the other hand, we see in the phase plane two heteroclinic trajectories connecting the critical points $O_1$ to $O_2$ of system \eqref{Eq:Dyn:Intro}. The homoclinic trajectory corresponds to the usual critical exponent $q=q^*(k)$ which is analysed in our second result, see Theorem \ref{Main:2:Intro} below.

\medbreak

Finally, the number of solutions to ($P_\lambda$) is in exact correspondence with the number of intersections between a suitable horizontal line and the orbits  (a spiral or stable node) of system \eqref{Eq:Dyn:Intro}. From this, we can construct a unique, large finite or infinite number of solutions to ($P_\lambda$) which is achieved following basically the arguments in \cite{JoLu73}.

\medbreak

We now consider the critical exponent problem
\begin{equation}\label{Eq:critical}
\begin{cases}
S_k(D^2u)= \lambda(1- u)^{q^*(k)}&\mbox{in }\;\; B,\\
u <0 & \mbox{in }\;\; B,\\
u=0 &\mbox{on }\; \partial B,
\end{cases}
\end{equation}
with $q^*(k)$ as in \eqref{q:ast:k}. Since we are looking for radially symmetric solutions to \eqref{Eq:critical}, we introduce a family of Bliss functions
\begin{equation}\label{Eq:FS}
w_d(x) = -\frac{\left[d \binom{n}{k}^\frac{1}{k} \left(\frac{n-2k}{k}\right) \right]^{\frac{n-2k}{2(k+1)}}}{\left( 1 + d|x|^2\right)^{\frac{n-2k}{2k}}},\;\; x\in \RR^n,
\end{equation}
for $d>0$ (see \cite{Blis30} and \cite{ClFM96}) and then we write the solutions of \eqref{Eq:critical} in terms of the functions $w_d$ as in \eqref{Eq:FS}. For this, note that $w_d(\cdot)$ solves the problem
\begin{equation}\label{bliss:solution}
S_k(D^2u)= (- u)^{q^\ast(k)}\,\, \mbox{ in }\,\, \RR^n,
\end{equation}
for all $d>0$. Restricting the functions $w_d$ to the unit ball, we define:
\begin{equation}\label{Eq:wd:v}
v_{\lambda} = w_{d_-(\lambda)}|_{B}, \, V_{\lambda} = w_{d_+(\lambda)}|_{B},\, V^*=w_k|_{B},
\end{equation}
where $d_-(\lambda), \, d_+(\lambda)$ are the only two solutions of
\begin{equation}\label{Eq:d}
\lambda (d+1)^{k+1} - \binom{n}{k}\left(\frac{n-2k}{k}\right)^k d^k = 0,
\end{equation}
provided that $\lambda$ is less than or equal to $\frac{\binom{n}{k}(\frac{n-2k}{k})^k k^k}{(k+1)^{k+1}}$. When the parameter $\lambda$ is equal to this last value, we have only one solution $d=k$.

Now, we can state our second result.
\begin{theorem}\label{Main:2:Intro}
Let $n>2k$. Consider the functions $v_{\lambda}$, $V_{\lambda}$, and $V^*$  as in \eqref{Eq:wd:v}.  Set
\[
\mu^\ast(k) = \frac{\binom{n}{k}(\frac{n-2k}{k})^k k^k}{(k+1)^{k+1}}.
\]
\begin{itemize}
	\item[(i)] If $\lambda\in (0,\mu^\ast(k))$, then there exist exactly two solutions of \eqref{Eq:critical} given by
	\[
u_{\lambda}(x)=1-\lambda ^{-\frac{n-2k}{2k(k+1)}}(-v_{\lambda}(x)) ,\;\, U_{\lambda}(x)=	1-\lambda ^{-\frac{n-2k}{2k(k+1)}}(-V_{\lambda}(x)).
	\]
	\item[(ii)] If $\lambda = \mu^\ast(k)$, then \eqref{Eq:critical} has a unique solution given by
\[
u^\ast(x) = 1-(\mu^\ast(k))^{-\frac{n-2k}{2k(k+1)}} (-V^\ast(x)) = 1- \left(\frac{1+k}{1+k|x|^2}\right)^{\frac{n-2k}{2k}}.
\]		
\end{itemize}
\end{theorem}
We note that the value $\mu^*(k)$, hereafter denoted by $\mu^*$, is computed by the rescaling $-v = \lambda^{\frac{n-2k}{2k(k+1)}}(1-u)$. By applying the boundary condition on $v$, and noting that $v$ is a Bliss function by \cite[Lemma 5.2]{ClFM96}, we obtain $\mu^*$. Notice that, for $q=q^*(k)$, we have the estimate $\lambda^* \geq \mu^*$ by Theorem \ref{Main:3:Intro} (see below). We mention that, in the case $k=1$, the value $\mu{^\ast}=\frac{n(n-2)}{4}$ coincides with the extremal value obtained in the classical paper \cite{JoLu73}. See also \cite{GaMa02} and \cite{Isse09}. We point out that, if any solution of the problem
\begin{equation}\label{Rn:1:Intro}
S_k(D^2u)=(-u)^{q^\ast(k)}, \, u<0\,\, \mbox{ in }\,\,\RR^n
\end{equation}
is radially symmetric around the origin, then we have $\lambda^* = \mu^*$. The study of radially symmetric solutions to \eqref{Rn:1:Intro} in $\RR^n$ seems not to have been discussed in the literature, except if we restrict to a ball, see \cite{Bran13, Tso90}. However, for the Laplacian operator ($k=1$) and for the $p$-Laplacian operator, there exist well-known references about radially symmetric solutions to the corresponding problems, see e.g. \cite{Broc97, CaGS89, ChLi91, DMMS14, GiNN81}.

\medbreak

Our third result deals with the existence of a finite positive number $\lambda^\ast$, as in \eqref{lambda-ast}, such that problem $(P_{\lambda})$ has either negative maximal bounded solutions if $\lambda\in (0,\lambda^\ast)$, at least one integral (or weak) solution (possibly unbounded) when $\lambda=\lambda^\ast$, or no classical solutions when $\lambda>\lambda^\ast$. More precisely,
\begin{theorem}\label{Main:3:Intro}
Let $n > 2k$ and $q > k$. Then there exists $\lambda^\ast>0$ such that:
\begin{itemize}
	\item[(a)] If $\lambda \in (0, \lambda^\ast)$, then $(P_{\lambda})$ admits a maximal bounded solution.
	\item[(b)] If $\lambda=\lambda^\ast$, then $(P_{\lambda})$ admits at least one possible unbounded integral solution.
	\item[(c)] If $\lambda > \lambda^\ast$, then $(P_{\lambda})$ admits no classical solutions.
\end{itemize}
\end{theorem}

Note that in the case $k=1$, i.e., for the Laplace operator, it is well-known that there exists a finite positive (extremal) parameter $\lambda^*$ such that problem (\ref{Eq:f:pol}) has a positive minimal classical solution $u_{\lambda}\in C^2(\overline\Omega)$ if $0<\lambda<\lambda^*$, while no solution exists, even in the weak sense, for $\lambda>\lambda^*$. (See \cite{BCMR96} and the references therein). The boundedness of solutions for the extremal value $\lambda^\ast$, which are so called {\it extremal solutions}, depends on the dimension $n$ and on critical values of the power $q$. In \cite{BrVa97} the aim was to study the properties of the extremal solutions of problem (\ref{Eq:f:pol}). Further, in this reference the main interest centered on the unbounded solutions (i.e. singular solutions or blow-up solutions). For $n>2$ and $q>\frac{n}{n-2}$ we have the explicit weak solution
\[
U(x)=|x|^{-\frac{2}{q-1}}-1
\]
corresponding to the parameter
\begin{equation*}
\lambda^{\sharp}=\frac{2}{q-1}\left(n-\frac{2q}{q-1}\right).
\end{equation*}
Thus problem (\ref{Eq:f}) admits at least one solution if $\lambda\leq\lambda^{\sharp}$ (see\cite{BrVa97}).

In our case, for $n>2k$ and $q>\frac{nk}{n-2k}$, we introduce the radial function $U$ defined by
\begin{equation}\label{singusolu}
U(x)=1-|x|^{-\frac{2k}{q-k}}\;\;\;\forall x\in B\setminus \{0\}.
\end{equation}
This function is an explicit integral solution of problem (\ref{Eq:f:pol} for more details), corresponding to the parameter
\begin{equation*}
\lambda = c_{n,k}\tilde{\lambda}(k),
\end{equation*}
provided that $q\geq q_{JL}(k)$ holds. Note that, for $k=1$, we have $c_{n,1}=1$ and $\tilde{\lambda}(1)=\lambda^{\sharp}$.

\section{Proof of Theorem \ref{Main:3:Intro}}
To prove the statement (a) we need two lemmas. The first lemma is a variant of \cite[Lemma 4]{BCMR96} tailored to our needs.
\begin{lemma}\label{BCMR96}
Let $g$ be a $C^1$ convex and nondecreasing function on $[0,\infty)$. Assume that $g(0)>0$ and set
\[
h(s) = \int_s^0 \frac{1}{g(-t)}\,dt, \, s\leq 0.
\]
Let $\tilde{g}$ be a $C^1$ positive function on $[0,\infty)$ such that $\tilde{g}\leq g$ and $\tilde{g}'\leq g'$. Now Set
\[
\tilde{h}(s) = \int_s^0 \frac{1}{\tilde{g}(-t)}\,dt, \, s\leq 0,
\]
and $\Phi(s) = \tilde{h}^{-1}(h(s))$ for all $s\leq 0$. Then
\begin{itemize}
	\item[(i)] $\Phi(0) = 0$ and $s \leq \Phi(s) \leq 0$ for all $s\leq 0$.
	\item[(ii)] If $\lim_{s\to -\infty}h(s)$ exists and $\tilde{g} \neq g$, then $\lim_{s\to -\infty}\Phi(s)$ exists as well.
    \item[(iii)] $\Phi$ is increasing, convex and $\Phi'(s) \leq 1$ for all $s\leq 0$
\end{itemize}
\end{lemma}
\begin{proof}
Clearly $\Phi(0)=0$. Further, since $h$ and $\tilde{h}$ are decreasing functions, we have $s \leq \Phi(s) \leq 0$ for all $s\leq 0$. Thus (i) holds. Property (ii) is clear. To prove (iii), we have
\[
\Phi'(s)=\frac{\tilde{g}(-\Phi(s))}{g(-s)}>0,
\]
and
\[
\Phi''(s)=-\tilde{g(-\Phi(s))}\frac{\tilde{g}'(-\Phi(s))-g'(-s)}{g(-s)^2}\,\cdot
\]
Since $\tilde{g}(-\Phi(s)) \leq g'(-\Phi(s)) \leq g'(-s)$ it follows that $\Phi$ is a convex function on $(-\infty, 0]$, which completes the proof.
\end{proof}

\begin{lemma}\label{max:sol}
Let $n > 2k$, $q > k$ and $\lambda_0 >0$. Assume that there exists a classical solution of
\begin{equation}
\begin{cases}\label{Eq:2:0}
c_{n,k}r^{1-n}\left(r^{n-k}(w')^k \right)' = \lambda_0 (1-w)^q\,,\quad 0<r<1,\\
w  < 0 \,, \hspace{4.65cm} 0<r<1,\\
w'(0)  =0,\, w(1)=0, &
\end{cases}
\end{equation}

Then, for any $\lambda\in (0,\lambda_0)$, problem $(P_\lambda)$ has a maximal bounded solution. Moreover, the maximal solutions form a decreasing sequence as $\lambda$ increases.
\end{lemma}
\begin{proof}
Fix $\lambda\in (0, \lambda_0)$ and define the functions
\[
g(t) = \left[\lambda_0(1+t)^q\right]^{1/k} \text{ and } \tilde{g}(t) = \left[\lambda(1+t)^q\right]^{1/k},\text{ for all } t\geq 0.
\]
Set $\Phi(s) = \tilde{h}^{-1}(h(s))$ ($s\leq 0$) with $h$ and $\tilde{h}$ as in Lemma \ref{BCMR96}. Since $q>k$, we have that $\lim_{s\to -\infty}h(s)$ exists and hence $\Phi$ is bounded by Lemma \ref{BCMR96} (ii). Next, by \eqref{Eq:2:0} and Lemma \ref{BCMR96} we have
\begin{eqnarray*}
S_k(D^2 \Phi(w))&=&c_{n,k}kr^{1-k}(\Phi'(w)w')^{k-1}\left(\Phi''(w)(w')^2+\Phi'(w)w''+\frac{n-k}{k}\frac{\Phi'(w)w'}{r}\right)\\
&\geq & c_{n,k}kr^{1-k}(\Phi'(w))^k(w')^{k-1}\left(w''+\frac{n-k}{k}\frac{w'}{r}\right)\\
& = &(\Phi'(w))^k S_k(D^2 w) = \frac{(\tilde{g}(-\Phi(w)))^k}{(g(-w))^k}S_k(D^2 w) = \lambda(1-\Phi(w))^q.
\end{eqnarray*}
Therefore,  $\Phi(w)$ is a bounded subsolution of $(P_\lambda)$ and hence by the method of super and subsolutions we have, by \cite[Theorem 3.3]{Wang94}, a solution $u\in L^{\infty}((0,1))$ of $(P_\lambda)$ with $\Phi(w) \leq u \leq 0$. Now, to prove that ($P_{\lambda}$) admits a maximal solution we consider $u_1$ as a solution of
\begin{equation*}
\begin{cases}
S_k(D^2 u_1)= \lambda &\mbox{in }\;\; B,\\
u_1=0 &\mbox{on }\; \partial B.
\end{cases}
\end{equation*}
Since $u$ is in particular a subsolution of ($P_{\lambda}$), we have $u\leq u_1$ on $B$ by the comparison principle, see \cite{TrXu99}. Next, we define $u_i$ ($i=2,3,\ldots$) as a solution of
\begin{equation*}
\begin{cases}
S_k(D^2 u_i)= \lambda(1-u_{i-1})^q &\mbox{in }\;\; B,\\
u_i=0 &\mbox{on }\; \partial B.
\end{cases}
\end{equation*}
Using again the comparison principle we obtain a decreasing sequence of $u_i$ bounded from below by $u$ and by 0 from above. Hence, we can pass to the limit and we obtain a solution $u_{max}$ of ($P_{\lambda}$), which is maximal since the recursive sequence $\{u_i\}$ does not depend on the subsolution $u$. Now, let $\lambda_1<\lambda_2$ and $u_{\lambda_1}$, $u_{\lambda_2}$ be maximal solutions of $(P_{\lambda_i})$ ($i=1,2$), respectively. Note that $u_{\lambda_2}$ is a subsolution of $(P_{\lambda_1})$, whence $u_{\lambda_2} \leq u_{\lambda_1}$ by the maximality of $u_{\lambda_1}$.
\end{proof}

Now, we show that $\lambda^\ast$ finite and positive. For $R>1$, let $B_R$ be a ball centered at zero with radius $R$ such that $\overline{B}\subset B_R$ and let $\eta$ be the solution of
\begin{equation*}
\begin{cases}
S_k(D^2\eta)= 1&\mbox{in }\;\; B_R,\\
\eta=0 &\mbox{on }\; \partial B_R.
\end{cases}
\end{equation*}
Then there exists a negative constant $\beta$ such that $\eta<\beta<0$ on $\partial{B}$.
Set $M=max_{x\in\overline{B}}\,|\eta|$ and take $\lambda<(1+M)^{-q}$. Then
\[
S_k(D^2\eta)=1>\lambda(1+M)^q\geq\lambda(1-\eta)^q\;\;\mbox{in}\;\;B.
\]
By \cite[Theorem 3.3]{Wang94}, for any $\lambda\in (0,(1+M)^{-q})$ there exists a solution $u_{\lambda}$ of $(P_\lambda)$. Hence $\lambda^\ast>0$.
To see that $\lambda^\ast$ is finite we consider the inequality
\begin{equation}\label{Eq:Nineq}
\Delta u\geq C(n,k)[S_k(D^2 u)]^\frac{1}{k},\; u\in \Phi^k(B).
\end{equation}
See e.g. \cite{Wang09}. Let $\lambda_1$ be the first eigenvalue of $-\Delta$ with zero Dirichlet boundary condition. Let $\lambda\in (0,\lambda^\ast)$ and let $u$ be a solution of problem $(P_\lambda)$. Then, using \eqref{Eq:Nineq}, we obtain
\[
\Delta u\geq C(n,k)[\lambda(1-u)^q]^\frac{1}{k}=C(n,k)\lambda^{\frac{1}{k}}(1-u)(1-u)^{\frac{q-k}{k}}>C(n,k)\lambda^{\frac{1}{k}}(1-u)>\lambda^{\frac{1}{k}}C(n,k)(-u),
\]
which in turn implies $\lambda<\left(\frac{\lambda_1}{C(n,k)}\right)^k$. Thus $\lambda^{\ast}$ is finite.

Now, let $\lambda\in (0,\lambda^{\ast})$. Then $u_\lambda$ is a maximal bounded solution of ($P_{\lambda}$) by Lemma \ref{max:sol} applied with $\lambda_0 \in (\lambda, \lambda^{\ast})$. This proves assertion (a).

\medbreak

Now let $\lambda_i$ be an increasing sequence such that $\lambda_i\rightarrow \lambda^\ast$ as $i\rightarrow +\infty$ and let $u_{\lambda_i}$ be a maximal solution of $(P_{\lambda_i})$. By Lemma \ref{max:sol}, for all $r\in [0,1]$, we have $u_{\lambda_{i+1}}(r)\leq u_{\lambda_{i}}(r)\leq 0$. On the other hand, integrating the equation in $(P_{\lambda_i})$, we obtain
\[
u_{\lambda_i}(r)=\int_{r}^{1}\left[c_{n,k}^{-1}\tau^{k-n}\int_{0}^{\tau}s^{n-1}\lambda_{i}(1-u_{\lambda_i}(s))^q\,ds\right]^{\frac{1}{k}}d\tau.
\]
Now, applying twice the monotone convergence theorem we conclude that
\[
u^\ast(r):=\lim_{i\rightarrow +\infty}u_{\lambda_i}(r), \;\;\; \text{ exists a.a. } r\in  (0,1)
\]
and
\[
u^\ast(r)=\int_{r}^{1}\left[c_{n,k}^{-1}\tau^{k-n}\int_{0}^{\tau}s^{n-1}\lambda^\ast(1-u^\ast(s))^q\,ds\right]^{\frac{1}{k}}d\tau, \;\;\; \text{ a.a. } r\in  (0,1).
\]
This proves assertion (b).

Assertion (c) follows immediately by definition of $\lambda^\ast$.

\section{Proof of Theorem \ref{Main:2:Intro}}

The rescaled function $v=-\lambda^{\frac{n-2k}{2k((k+1))}}(1-u)$ solves the problem
\begin{equation}\label{Eq:critical:v}
\begin{cases}
S_k(D^2v)= (- v)^{q^\ast(k)}&\mbox{in }\;\; B,\\
v = -\lambda^{\frac{n-2k}{2k((k+1))}} &\mbox{on }\; \partial B,\\
v > -\lambda^{\frac{n-2k}{2k((k+1))}} &\mbox{in }\;  B.
\end{cases}
\end{equation}
In particular, in radial coordinates $v=v(r)$, the indicates function solves the ordinary differential equation
\begin{equation}\label{Eq:radial:critical}
c_{n,k}\left(r^{n-k}(v')^k \right)' = r^{n-1}(-v)^{q^\ast(k)}
\end{equation}
with the boundary conditions $v'(0) = 0$ and $v(1) = -\lambda^{\frac{n-2k}{2k(k+1)}}$. Now define $(0,R)$, with $0<R\leq \infty$, as the maximal continuation interval for $v(r)$ under the restriction
\[
v(r)<0\, \text{ for all } r\in (0,R).
\]
Further, after integration of \eqref{Eq:radial:critical} over $(0,r)$, with $r\in (0,R)$, we obtain
\begin{equation}\label{v:prima:critical}
v'(r) >0\, \text{ for all } r\in (0,R).
\end{equation}
We claim that $R=\infty$. Suppose by contradiction that $R<\infty$. Since $v$ is increasing by \eqref{v:prima:critical} and bounded from above by $0$, we have
\[
v(R) =\lim_{r\to R} v(r) = 0,
\]
by the maximality of $R$. Therefore $v$ is a nontrivial solution of
\begin{equation*}
\begin{cases}
S_k(D^2v)= (- v)^{q^\ast(k)}&\mbox{in }\;\; B_R,\\
v = 0 &\mbox{on }\; \partial B_R,
\end{cases}
\end{equation*}
which contradicts Tso's nonexistence result, see \cite[Proposition 1]{Tso90}.

Next, after scaling
\[
s=\left(c_{n,k}\right)^{-\frac{1}{2k}}r, \,\, r>0
\]
we can rewrite \eqref{Eq:radial:critical} as
\begin{equation}\label{Eq:resc:v}
\left(s^{n-k}(v')^k \right)' = s^{n-1}(-v)^{q^\ast(k)}, \, s>0.
\end{equation}
Choosing $\alpha=n-k$, $\beta=k-1$, and $\gamma=n-1$ in \cite[Equation (1.12)]{ClFM96} we conclude that \eqref{Eq:resc:v} admits a unique solution of the form
\begin{equation}\label{Eq:FS:1}
w_d(x) = -\frac{\left[d \binom{n}{k}^\frac{1}{k} \left(\frac{n-2k}{k}\right) \right]^{\frac{n-2k}{2(k+1)}}}{\left( 1 + d|x|^2\right)^{\frac{n-2k}{2k}}},\,\, x\in \RR^n,
\end{equation}
for $d>0$ (see \cite{Blis30} and \cite{ClFM96}). Now we can write the solutions of \eqref{Eq:critical:v} in terms of the functions $w_d$ as \eqref{Eq:FS:1}. More precisely, it is easy to see that $w_d$ is a solution of \eqref{Eq:critical:v} if, and only if, the there exists a value $d>0$ such that
\begin{equation}\label{Eq:d}
\lambda (d+1)^{k+1} - \binom{n}{k}\left(\frac{n-2k}{k}\right)^k d^k = 0.
\end{equation}
We can now verify by elementary calculus that \eqref{Eq:d} has either a unique solution $d = k$ if $\lambda = \mu^*$, exactly two solutions $d_-(\lambda) <d_+(\lambda)$ if $0<\lambda < \mu^*$ and no solutions $d$ if $\lambda > \mu^*$. Hence problem \eqref{Eq:critical:v} has a solution if, and only if, $\lambda\leq \mu^*$.

If $\lambda< \mu^*$, there exist exactly two solutions of \eqref{Eq:critical:v} given by
\begin{equation*}
v_{\lambda} = w_{d_-(\lambda)}|_{B}, \, V_{\lambda} = w_{d_{+}(\lambda)}|_{B}.
\end{equation*}
On the other hand, if $\lambda = \mu^*$, we have a unique solution of \eqref{Eq:critical:v} given by
\begin{equation*}
V^*=w_k|_{B}.
\end{equation*}

Therefore, by the rescaled function $v$, we conclude that problem \eqref{Eq:critical} has exactly two solutions $u_\lambda$, $U_\lambda$ if $\lambda< \mu^*$ and a unique solution $u^*$ if $\lambda=\mu^*$, where
	\[
u_{\lambda}(x)=1-\lambda ^{-\frac{n-2k}{2k(k+1)}}(-v_{\lambda}(x)) ,\;\, U_{\lambda}(x)=	1-\lambda ^{-\frac{n-2k}{2k(k+1)}}(-V_{\lambda}(x))
	\]
and
\[
u^\ast(x) = 1-(\mu^\ast)^{-\frac{n-2k}{2k(k+1)}} (-V^\ast(x)) = 1- \left(\frac{1+k}{1+k|x|^2}\right)^{\frac{n-2k}{2k}}.
\]		
The proof is now complete.

\section{Proof of Theorem \ref{Main:1:Thm1}}

Let $n>2k$ and $q>q^*(k)$. Set $A=u(0) < 0$ and introduce the nonnegative variable
\begin{equation}\label{def:s:r}
s=\left(\frac{\lambda(1-A)^{q-k}}{c_{n,k}\tilde{\lambda}(k)} \right)^{\frac{1}{2k}}r,\,\: \text{for all }r\geq 0
\end{equation}
with $\tilde{\lambda}(k)$ as in \eqref{Lambda:Tilde:Intro}. Hereafter $\tilde{\lambda}(k)$ is denoted by $\tilde{\lambda}$.
Note that $\tilde{\lambda}>0$ since $q > q^*(k)$. Define the rescaled function $v(s)$ as
\begin{equation}\label{def:v:s}
v(s) = -\frac{1-u}{1-A}.
\end{equation}
We note that, by $(P_{\lambda})$ and \eqref{def:v:s}, $v$ solves the initial value problem
\begin{equation}\label{Eq:v}
\begin{cases}
\left(s^{n-k}(v')^k\right)'= \tilde{\lambda} s^{n-1}(-v)^q, & s>0,\\
v(0) = -1,\\
v'(0) =0.
\end{cases}
\end{equation}

\begin{lemma}\label{Uniq:sol}
Let $q\geq q^\ast(k)$. Then there exists a unique global solution $v$ of \eqref{Eq:v} in the regularity class $C^2(0,\infty)\cap C^1[0,\infty)$.
\end{lemma}
\begin{proof}
Let $\tau = \tilde{\lambda}^{\frac{1}{2k}}s$ and  set $w(\tau) = v(s)$. Then we may rewrite \eqref{Eq:v} as
\begin{equation}\label{Eq:w:aux}
\begin{cases}
\left(\tau^{n-k}(w')^k\right)'= \tau^{n-1}(-w)^q, & \tau>0,\\
w(0) = -1,\\
w'(0) =0.
\end{cases}
\end{equation}
Defining $B(r)=\int_0^r s^{\frac{(\alpha-\beta)(q+1)}{\beta+1}}(a(s)s^{\theta})'ds$, $r>0$, with $\alpha=n-k$, $\beta=k$, $\gamma=n-1$, $a(s)\equiv 1$ and $\theta=[(\gamma+1)(\beta+1)-(\alpha-\beta)(q+1)]/(\beta +1)$, we obtain $B(r)\leq 0$ for $r>0$ if, and only if, $q\geq q^*(k)$. Then the global existence of \eqref{Eq:w:aux} follows from \cite[Theorem 4.1]{ClMM98}. The uniqueness follows from the contraction mapping theorem as we show next. We use the notation given in \cite[Theorem 4.1]{ClMM98} and define the map $T: \mathcal{B} \to  \mathcal{B}$, where $ \mathcal{B}:=\{ w\in C[0,t_0]\;:\; w(0)=-1 \text{ and } |w+1| \leq 1/2\}$, by
\[
T(w)(r):= -1 + \int_0^r \left( \frac{1}{t^{n-k}}\int_0^t s^{n-1}(-w(s))^q ds\right)^{\frac{1}{k}}dt,\; r\in [0,t_0].
\]
Using the arguments given in the proof of \cite[Theorem 4.1]{ClMM98}, we see that the map $T$ satisfies $T(\mathcal{B}) \subset \mathcal{B}$ and admits a fixed point. We show now that $T$ is a contraction. To this end, let $w_1, w_2 \in \mathcal{B}$ be fixed, and define
\[
h_i(t)=\frac{1}{t^{n-k}}\int_0^t s^{n-1}(-w_i(s))^q ds, \; i=1,2.
\]
It is easy to see that
\[
h_i(t) \geq \left(\frac{3}{2}\right)^q \frac{t^k}{n}, \; i=1,2.
\]
On the other hand, from the inequality
\[
|a^{\frac{1}{k}} - b^{\frac{1}{k}}| \leq \frac{2}{a^{\frac{k-1}{k}} + b^{\frac{k-1}{k}}}|a-b|, \; a,b >0,
\]
we conclude that there exists a constant $c>0$ such that
\[
|h_1(t)^{\frac{1}{k}} - h_2(t)^{\frac{1}{k}}| \leq c \, t^{1-k} |h_1(t)-h_2(t)|,\;\; \text{ for all}\;\; t \in (0,r).
\]

Hence we have the estimate
\begin{align*}
|T(w_1)(r) - T(w_2)(r)| &\leq c \int_0^r t^{1-k}|h_1(t)-h_2(t)| dt\\
& \leq \tilde{c} r^2 \sup_{s\in [0,t_0]}|w_1(s)-w_2(s)|,
\end{align*}
since
\[
|(-w_1(s))^q - (-w_2(s))^q | \leq c_0 |w_1(s) - w_2(s)|
\]
for some positive constant $c_0$ and for all $s\in [0,t_0]$. We choose $r$ such that $\tilde{c} r^2 <1$, which ensures the existence of a unique local solution. Using standard arguments of continuation on the maximal existence interval, we deduce the existence of a unique global solution on $[0,\infty)$.
\end{proof}

Next, we introduce a dynamical system associated to \eqref{Eq:v} by means of the following change of variables of Emden-Fowler type:
\begin{equation}\label{EF:var}
s=e^t, \text{ for all } s>0
\end{equation}
and the introduction of the functions
\begin{equation}\label{y-z:def}
y(t) = \left(\frac{dv}{dt}\right)^k e^{\frac{2k^2}{q-k}t} \;\;\;\text{and}\;\;\; z(t) = \tilde{\lambda} e^{\frac{2kq}{q-k}t} (-v(e^t))^q.
\end{equation}
After a straightforward (and lengthy) computation, we conclude that the pair $(y(t), z(t))$ satisfies the following system of differential equations
\begin{equation}\label{Eq:Dyn}
\begin{cases}
\frac{dy}{dt}= z(t) - \frac{a}{q-k} y(t),\\
\,\\
\frac{dz}{dt} = \frac{2kq}{q-k} z(t) - q\tilde{\lambda}^{\frac{1}{q}}y(t)^{\frac{1}{k}}z(t)^{1-\frac{1}{q}},\\
\end{cases}
\end{equation}
where $a$ is defined as in \eqref{def:a:Intro}. Moreover, the function $z = z(t)$ satisfies
\begin{equation}\label{lim:z}
\lim_{t\to -\infty}z(t)e^{-\frac{2kq}{q-k}t} = \tilde{\lambda}
\end{equation}
by \eqref{Eq:v},  \eqref{EF:var} and \eqref{y-z:def}. This in turn implies that $z$ vanishes as $t\to -\infty$. Analogously, we can obtain, by \eqref{Eq:v},  \eqref{EF:var} and \eqref{y-z:def} that
\begin{equation}\label{lim:y:0}
\lim_{t\to -\infty}y(t) = 0.
\end{equation}

\begin{remark} For $q\geq q^*(k)$, the unique global solution of \eqref{Eq:v} corresponds, by Lemma \ref{Uniq:sol}, to a global solution $(y,z)$ of \eqref{Eq:Dyn}-\eqref{lim:z}. On the other hand, a global solution $(y,z)$ of \eqref{Eq:Dyn}-\eqref{lim:z} defines, by Lemma \ref{Uniq:sol} and reversing the definitions of \eqref{EF:var}-\eqref{y-z:def}, a unique solution $v$ of \eqref{Eq:v}. Thus we have a one-to-one correspondence between solutions of \eqref{Eq:v} and solutions of \eqref{Eq:Dyn}-\eqref{lim:z}.
\end{remark}

\begin{lemma}
Let $q\geq q^*(k)$. The unique global solution $v$ of \eqref{Eq:v} is globally bounded if, and only if, $\lim_{t\to -\infty} (y(t), z(t))=(0,0)$.
\end{lemma}
\begin{proof}
We have already shown the necessity of the condition by \eqref{lim:z}-\eqref{lim:y:0}. Now assume that
\[
\lim_{t\to -\infty} (y(t), z(t))=(0,0).
\]
From the analysis of the field directions to \eqref{Eq:Dyn} we see that, for $t$ small enough, $y'>0$ and $z'>0$, which may be used to estimate $y(t)$ and $z(t)$. To this end, we combine the inequalities arising from $y'>0$ and $z'>0$ and obtain
\begin{equation}\label{dz:1}
\frac{dz(t)}{dt} >\frac{2kq}{q-k}z(t) - q\tilde{\lambda}^{\frac{1}{q}}\left(\frac{q-k}{a}\right)^{\frac{1}{k}}z(t)^{1+\frac{q-k}{qk}} >0, \text{ for all } t<\bar{t},
\end{equation}
where $\bar{t}<0$ is chosen small enough. Next, by \eqref{dz:1}, we may integrate over $(t,\bar{t})$ the quotient
\[
\int_t^{\bar{t}} \frac{z(\tau)^{-1-\frac{q-k}{qk}}z'(\tau)}{\frac{2kq}{q-k}z(\tau)^{-\frac{q-k}{qk}} - q\tilde{\lambda}^{\frac{1}{q}}\left(\frac{q-k}{a}\right)^{\frac{1}{k}}}\,d\tau >\bar{t}-t >0.
\]
After the change of variable $\eta(\tau)=\frac{2kq}{q-k}z(\tau)^{-\frac{q-k}{qk}} - q\tilde{\lambda}^{\frac{1}{q}}\left(\frac{q-k}{a}\right)^{\frac{1}{k}}$, we obtain the estimate
\[
z(t) \leq Ce^{\frac{2kq}{q(k+1)-k} t}\leq Ce^{\frac{2kq}{q-k} t}.
\]
Thus, by definition of $z$ in \eqref{y-z:def}, we have
\[
0>v(s) \geq -\left(\frac{C}{\tilde{\lambda}}\right)^{\frac{1}{q}}, \text{ for all } s\in (0,e^{\bar{t}}).
\]
Since $v$ is increasing by the equation in \eqref{Eq:v}, we obtain that $v$ is globally bounded. The initial condition $v'(0)=0$ follows by  L'Hospital's rule and the equality $v(0)=-1$ follows from \eqref{lim:z}.
\end{proof}

The dynamical system \eqref{Eq:Dyn} has exactly two stationary points in the phase plane $(y,z)$. We denote these two points by $O_1$ and $O_2$, where
\[
O_1= (0,0)\;\;\;\text{ and }\;\;\; O_2 =\left(\frac{q-k}{a}\tilde{\lambda}, \tilde{\lambda}\right).
\]
Now we consider the linearization at the equilibrium point $O_2$ of \eqref{Eq:Dyn}. The Jacobian matrix at $O_2$ is given by
\[ J=
\left(\begin{array}{cc}
\frac{-a}{q-k}& 1\\
\,\\
\frac{-2qa}{(q-k)^2} & \frac{2k}{q-k}
\end{array}
\right).
\]
The eigenvalues of $J$ are
\[
\lambda_{\pm} =\frac{1}{2}\, \text{tr} J \pm \frac{1}{2} \sqrt{(\text{tr} J)^2 - 4 \text{det} J}
\]
where $\text{tr} J = \frac{2k-a}{q-k}$ and $\text{det} J = \frac{2a}{q-k}$. The discriminant is given by
\[
\varDelta(a)=(q-k)^{-2}[(2k-a)^2-8a(q-k)].
\]
The location on the complex plane of the eigenvalues $\lambda_{\pm}$ is determined as follows:
\begin{itemize}
	\item[(i)] If $2k-a>0$ and $\varDelta >0$, then the eigenvalues $\lambda_{\pm}$ are real positive numbers.
	\item[(ii)] If $2k-a>0$ and $\varDelta <0$, then the eigenvalues $\lambda_{\pm}$ are complex numbers with positive real part.
	\item[(iii)] If $2k-a<0$ and $\varDelta <0$, then the eigenvalues $\lambda_{\pm}$ are complex numbers with negative real part.
	\item[(iv)] If $2k-a<0$ and $\varDelta >0$, then the eigenvalues $\lambda_{\pm}$ are negative real numbers.
	\item[(v)] If $2k-a=0$, then the eigenvalues $\lambda_{\pm}$ are purely imaginary.
\end{itemize}

We observe that the cases (i) and (ii) can be ignored since $2k-a>0$ is equivalent to $q<q^{\ast}(k)$. The condition $2k-a=0$ in (v) is equivalent to setting $q =q^*(k)$, and this case was already discussed in Theorem \ref{Main:2:Intro}. Additionally, we note that the orbit $(y(t),z(t))$ of \eqref{Eq:Dyn}-\eqref{lim:z} in the critical case $q =q^*(k)$ convergences to $O_1$ after one loop around $O_2$. Indeed, let $v$ the unique solution of \eqref{Eq:v}. Then, using Tso's nonexistence result and the arguments in the proof of Theorem \ref{Main:2:Intro}, we conclude that $v$ solves the problem
\begin{equation*}
\begin{cases}
S_k(D^2u)= \tilde{\lambda}(-v)^{q^*(k)} &\mbox{in }\;\; \RR^n,\\
v(0)= -1. &
\end{cases}
\end{equation*}
Then, after rescaling $w=\tilde{\lambda}^{\frac{n-2k}{2k(k+1)}} v$, we see that $w$ solves the problem \eqref{bliss:solution} for some $d$. The initial condition $v(0)=-1$ yields the existence of three values of $d$ as in the proof of Theorem \ref{Main:2:Intro}. Since $w$ is a Bliss function, we can compute explicitly $v$, which is given by
\[
v(s)=-\left(1+ds^2\right)^{-\frac{n-2k}{2k}}.
\]
Hence, by the change of variable \eqref{EF:var} and \eqref{y-z:def}, we have
\begin{align*}
y(t) & =\left[\frac{(n-2k)d}{k}\right]^ke^{\frac{(n+2)k}{k+1}t}\left(1+de^{2t}\right)^{-\frac{n}{2}}\\
z(t) & = \tilde{\lambda}e^{\frac{(n+2)k}{k+1}t} \left(1+de^{2t}\right)^{-\frac{n+2}{2}}.
\end{align*}
Therefore, $\lim_{t\to +\infty}(y(t),z(t)) = O_1$ since $n>2k$.

\begin{lemma}\label{Asym:z:1}
Let $n>2k$ and $q\geq q^*(k)$. Consider the unique solution $(y(t),z(t))$ of \eqref{Eq:Dyn}-\eqref{lim:z}. Then
\[
z(t) = ny(t) + o(y(t)), \, \text{ as } t\to -\infty.
\]
\end{lemma}
\begin{proof}
Let $z=z(t)$ and $v=v(t)$ be as in \eqref{y-z:def}. Then, by \eqref{lim:z}, we have
\begin{equation}\label{Asym:z}
z(t) \sim \tilde{\lambda} e^{\frac{2kq}{q-k}t},\, \text{ as } t\to -\infty.
\end{equation}
On the other hand, by \eqref{Eq:v}, \eqref{EF:var}, \eqref{y-z:def} and L'Hospital's rule, we deduce that
\[
\lim_{t\to -\infty}y(t)e^{-\frac{2kq}{q-k}t}= \frac{\tilde{\lambda}}{n},
\]
which in turn implies that
\begin{equation}\label{Asym:y}
y(t) \sim \frac{\tilde{\lambda}}{n}\,e^{\frac{2kq}{q-k}t},\, \text{ as } t\to -\infty.
\end{equation}
Thus the lemma follows immediately by combining \eqref{Asym:z} and \eqref{Asym:y}.
\end{proof}

\begin{lemma}\label{lem:qJL:nodo}
Let $n>2k+8$ and $q \geq q_{JL}(k)$. Then the unique solution $(y(t),z(t))$ \eqref{Eq:Dyn}-\eqref{lim:z} coincides with the graph of an increasing function $z = z(y)$ and
\begin{equation}\label{lim:y-z:o2}
\lim_{t\to +\infty} (y(t), z(t)) = O_2=:(y_2,z_2).
\end{equation}
\end{lemma}
\begin{proof}
We first describe the behavior of the orbit $(y(t),z(t))$ near $O_2$. To this end, we compute
\[
\gamma:=\lim_{t\to +\infty}\frac{z'(t)}{y'(t)}
\]
by L'Hospital's rule. The existence of the limit above is equivalent to solving the equation
\begin{equation}\label{gamma:def}
\gamma^2 -\frac{2k+a}{q-k}\gamma +\frac{2aq}{(q-k)^2}=0,
\end{equation}
whose roots are given by
\[
\gamma_{\pm} = \frac{2k+a}{2(q-k)}\pm \frac{1}{2(q-k)}\sqrt{(2k+a)^2-8aq}.
\]
Note that $(2k+a)^2-8aq \geq 0$ if, and only if, $q\geq q_{JL}(k)$. Thus, $\gamma_{\pm}$ are positive roots.

Next, consider a function $g$ defined by
\begin{equation}\label{Def:g}
g(y) = cy^{\beta}, \, y\geq 0,
\end{equation}
where the constants $c$ and $\beta$ are chosen such that the graph of $g$ connects the points $O_1$ and $O_2$. Using $z_2=g(y_2)$, we obtain $c=\tilde{\lambda} \left(\frac{2k}{q-k}\right)^{-\beta k}$. Now, setting $g'(y_2)=(\gamma_{+}+\gamma_{-})/2$, we have $\beta=\frac{2k+a}{2 a}$. Note that $\beta\in (1/2, 1)$ since $q\geq q_{JL}(k)$ and $2k-a=(n-2k)(q^*(k)-q)<0$.

We claim that the unique solution $(y(t), z(t))$ lies below $(y,g(y))$ when $y\in (0,y_2)$. Indeed, by Lemma \ref{Asym:z:1}, $(y(t), z(t))$ lies above the line $z=\frac{a}{q-k}\,y$ and below $(y,g(y))$ near $O_1$ on the phase plane. On the other hand, since the slope of the line $z=\frac{a}{q-k}\,y$ at $O_2$ is larger than $(\gamma_{+}+\gamma_{-})/2$ and the parabola in \eqref{gamma:def} is positive at $a/(q-k)$, we conclude that the orbit arrives at $O_2$ from above the line $z=\frac{a}{q-k}\,y$ near $O_2$.

Suppose now by contradiction that $(y(t), z(t))$ intersects the curve $(y,g(y))$ in a point $(y_0,z_0)$ at $t=t_0$. In this case we have two possibilities: the orbit  $(y(t), z(t))$ remains on the graph of $g$ for all $t\geq t_0$ and then $(y(t), z(t))$ arrives at $O_2$ with the same slope that $(y,g(y))$ at $y_2$, which is impossible since  $g'(y_2)=(\gamma_{+}+\gamma_{-})/2 \neq \gamma_{\pm}$ because the inequality $q>q_{JL}(k)$. The other case is that the orbit crosses the graph of $g$ at a point $y_0$. In this case we have
\begin{equation}\label{g:prima}
g'(y_0)< \frac{z'(t_0)}{y'(t_0)}.
\end{equation}
We now show that \eqref{g:prima} is impossible. To this end, define the functions
\begin{equation}\label{F_1:F_2}
\begin{cases}
F_1(y,z)= z - \frac{a}{q-k} y,\\
\,\\
F_2(y,z) = \frac{2kq}{q-k} z - q\tilde{\lambda}^{\frac{1}{q}}y^{\frac{1}{k}}z^{1-\frac{1}{q}},\\
\end{cases}
\end{equation}
for all $(y,z)\in [0,y_2]\times [0,z_2]$. Note that
\begin{equation}\label{F:at:y0}
F_2(y_0,g(y_0)) - g'(y_0)F_1(y_0,g(y_0)) = F_2(y_0,z_0) - g'(y_0)F_1(y_0,z_0) >0.
\end{equation}
Next, set
\begin{equation}\label{h:def:F}
h(y) = F_2(y,g(y)) - g'(y)F_1(y,g(y)), \, y\in (0,y_2).
\end{equation}
Then, by \eqref{Def:g} and \eqref{F_1:F_2}, we have
\[
h(y) = \frac{2kq}{q-k}\,cy^{\beta} -qc^{1-\frac{1}{q}}\tilde{\lambda}^{\frac{1}{q}} y^{\frac{1}{k}+\beta(1-\frac{1}{q})}-c^2\beta y^{2\beta -1}+\frac{a\beta c}{q-k}\,y^{\beta}, \, y\in (0,y_2).
\]
Using \eqref{F_1:F_2}, we obtain
\begin{equation}{\label{h:prima:y2}}
h'(y_2) = - [g'(y_2)]^2 + \frac{2k+a}{q-k}\,g'(y_2) - \frac{2aq}{(q-k)^2}\geq 0
\end{equation}
since $q\geq q_{JL}(k)$ and the inequality is strict if $q > q_{JL}(k)$. Now note that $h(0)=0$ since $\beta\in(1/2,1)$ and $h(y_2)=0$ by \eqref{lim:y-z:o2} and the equality $z_2=g(y_2)$. The derivative of $h$ may be written as
\begin{equation}\label{hprima}
h'(y)= y^{\beta-1}\rho(y),
\end{equation}
where the function $\rho$ is given by
\[
\rho(y)= \frac{\beta c}{q-k}\left(2kq+a\beta\right)  - q c^{1-1/q}\tilde{\lambda}^{1/q}\left[\frac{1}{k} + \beta\left(1-\frac{1}{q}\right)\right]y^{1/k - \beta/q} -\beta (2\beta-1)c^2y^{\beta-1}.
\]
Now, in order to determine the growth of $h$ around $0$, we rewrite $h'$ as $h'(y)=y^{1-\beta}\rho(y)/y^{2(\beta -1)}$ and conclude that
\begin{equation}\label{h:prima:0}
\lim_{y\to 0^+} h'(y) = -\infty,
\end{equation}
since $1/k - \beta/q >0$ and $\beta\in(1/2,1)$. Note that $\rho(y_2)\geq 0$ by \eqref{h:prima:y2}-\eqref{hprima}. In particular, $\rho(y_2)=0$ in the case $q=q_{JL}(k)$ by \eqref{h:def:F} and the equalities $g'(y_2)=\gamma_+=\gamma_- = \gamma$. On the other hand,  it is easy to see that
\begin{equation}\label{rho:0}
\lim_{y\to 0^+} \rho(y) = -\infty.
\end{equation}
Computing, $\rho'(y)$ we conclude that the equation $\rho'(y)=0$ admits only one solution which is not a minimum by \eqref{rho:0}. Further, the equation $\rho(y)=0$ has at most two solutions on $(0,y_2]$ since, if there is no solution, then $\rho<0$ on $(0,y_2]$ by \eqref{rho:0}. This would in turn imply that $h$ is decreasing on $(0,y_2]$ by \eqref{hprima}, which is a contradiction since $h(0)=h(y_2)=0$. Now, if we have two solutions on $(0,y_2)$ then $\rho(y_2)<0$, which is a contradiction since $\rho(y_2)\geq 0$. Hence the equation $\rho(y)=0$ admits at most two solutions on $(0,y_2]$. In this case we see that $h$ admits at most two critical points on $(0,y_2]$. But this is impossible since $h(y_0)>0$ by \eqref{F:at:y0}-\eqref{h:def:F} and \eqref{h:prima:y2} and \eqref{h:prima:0}. The proof is now  complete.
\end{proof}

\begin{lemma}\label{qk:espiral}
Let $q^*(k)<q<q_{JL}(k)$. Let $(y(t),z(t))$ be the unique solution of \eqref{Eq:Dyn}-\eqref{lim:z} and $v$ the unique solution of \eqref{Eq:v}. Then $(y(t),z(t))$ loops around $O_2$ an infinite number of times and converges to $O_2$ as $t\to +\infty$. Moreover, we have the following asymptotic behavior for $v$ at infinity
\begin{equation}\label{asymptotic}
v(s)\sim -s^{-\frac{2k}{q-k}}\;\;\mbox{as}\;\;s\to +\infty.
\end{equation}
\end{lemma}
\begin{proof}
We first show that the trajectory does not converge to $O_1$ as $t\to +\infty$. Suppose by contradiction that $(y(t),z(t))$ converges to $O_1$ as $t\to +\infty$. Then, for all $t$ large enough, $y(t)$ and $z(t)$ are decreasing functions. Hence, from \eqref{Eq:Dyn} we have
\begin{equation}\label{boundzeta}
z(t)\leq Cy(t)^\frac{q}{k}\qquad\forall t\geq \bar{t},
\end{equation}
where $\bar{t}$ is large enough and $C$ is a positive constant. Since $y,z$ are two vanishing functions as $t\to +\infty$ and $q>k$, by (\ref{Eq:Dyn}) and (\ref{boundzeta}), we conclude that
\begin{equation}\label{asympyprima}
y'(t)=-\frac{a}{q-k}\,y(t)+o(y(t))\;\;\mbox{as}\;\; t\to +\infty.
\end{equation}
Using (\ref{asympyprima}), we deduce that, for any $\epsilon >0$, there exists $T_\epsilon >0$ such that
\[
y(t)\leq C_\epsilon\,e^{\left(-\frac{a}{q-k}+\epsilon\right)t}\;\; \forall t>T_\epsilon
\]
and this, together with (\ref{y-z:def}) and (\ref{boundzeta}), yields
\begin{equation}\label{bound-v}
-v(s)\leq C_\epsilon s^{-\frac{n-2k}{k}+\epsilon}\;\;\forall s>S_\epsilon
\end{equation}
where $S_\epsilon$ is a constant depending on $\epsilon$. Choosing $\epsilon$ small enough in (\ref{bound-v}), we see that the function $s^{n-1}(-v(s))^q$ belongs to $L^1(0,\infty)$ and then by (\ref{Eq:v}) we obtain
\begin{equation}\label{simvprima}
v'(s)\sim \left(\tilde{\lambda}
\int_{0}^{\infty}\xi^{n-1}(-v(\xi))^q\,d\xi\right)^{\frac{1}{k}}s^{-\frac{n-k}{k}}\;\; \mbox{as}\;\;t\to +\infty.
\end{equation}
In turn now by (\ref{simvprima}) and L'Hospital's rule we obtain
\begin{equation}\label{simv}
v(s)\sim -\frac{k}{n-2k}\left(\tilde{\lambda}\int_{0}^{\infty}\xi^{n-1}(-v(\xi))^q\,d\xi\right)^\frac{1}{k}s^{-\frac{n-2k}{k}}\;\;\mbox{as}\;\; s\to +\infty.
\end{equation}
Next, choosing $a=\frac{n-2k}{k+1},\,b(r)=r,\,\alpha=n-k,\,\beta=k-1$ and $F(r,v)=-\tilde{\lambda}r^{n-1}\,\frac{(-v)^{q+1}}{q+1}$ in \cite[Proposition 4.1]{ClFM96}, we conclude that, for each $R>0$, the integrals on the left hand side in \cite[Proposition 4.1]{ClFM96} are zero. Consequently
\begin{align}\label{identityPPS}
& \frac{\tilde{\lambda}(n-2k)(q^*(k)-q)}{(k+1)(q+1)} \int_{0}^{R}\xi^{n-1}(-v(\xi))^{q+1}d\xi  \nonumber\\
&  = \frac{n-2k}{k+1}\,R^{n-k}v(R)(v'(R))^k +\frac{k}{k+1}\,R^{n-k+1}(v'(R))^{k+1} + \frac{\tilde{\lambda}}{q+1}\,R^n(-v(R))^{q+1}.
\end{align}
Using \eqref{simvprima}-\eqref{simv}, we deduce that the first two terms on the right hand side of (\ref{identityPPS}) go to zero as $R\to \infty$ since $n>2k$, and the last term vanishes as $R\to \infty$ since $q>q^*(k)$, which in turn yields
\[
\int_{0}^{\infty}\xi^{n-1}(-v(\xi))^{q+1}d\xi=0.
\]
This is a contradiction.

Therefore, the orbit $(y(t),z(t))$ converges to the stationary point $O_2$ or to a limit cycle. The existence of a limit cycle is excluded by the generalized version of Bendixson's Theorem. Indeed, define
\[
\psi(y,z) = z^{\frac{a}{2kq} -1},\, y>0,\, z>0.
\]
Then
\begin{align*}
\partial_y \left(\psi(y,z)\frac{d y}{dt}\right) + \partial_z \left(\psi(y,z)\frac{d z}{dt}\right) =& - \psi(y,z){\left(\frac{a}{2k} -1\right)\tilde{\lambda}^{\frac{1}{q}}y^{\frac{1}{k}} z^{-\frac{1}{q}}}<0,
\end{align*}
Since $q>q^*(k)$. Thus there is no limit cycle of \eqref{Eq:Dyn} by the Bendixson-Dulac Theorem. Hence $(y(t),z(t))$ converges to $O_2$ as $t\to +\infty$ after infinitely many loops around $O_2$ since the eigenvalues of the linearized system in $O_2$ are complex numbers in view of $q<q_{JL}(k)$. The asymptotic estimate (\ref{asymptotic}) is an immediate consequence of the change of variables \eqref{EF:var}-\eqref{y-z:def}.
\end{proof}

Next, returning to the proof of the theorem, we show the uniqueness and the multiplicity of solutions of $(P_{\lambda})$ as follows: let $s=s_0>0$ be fixed, define
\begin{equation}\label{As:lambdaS}
A(s_0)=1+(v(s_0))^{-1} \text{ and } \lambda(s_0)=\tilde{\lambda}\tau^{2k}(-v(s_0))^{q-k}
\end{equation}
and set $u_{\lambda(s_0)}(0)=A(s_0)$.  Since the component $z$ of the orbit satisfies $z(t) \to \tilde{\lambda}$ as $t\to +\infty$ for all $q>q^*(k)$ by Lemmas \ref{lem:qJL:nodo} and \ref{qk:espiral}, we have $A(s) \to -\infty$ as $s\to +\infty$ and $\lambda(s)\to \tilde{\lambda}$ as $s\to +\infty$ by the change of variable \eqref{EF:var}-\eqref{y-z:def}. Now, using the changes of variables \eqref{EF:var}-\eqref{y-z:def} together with \eqref{As:lambdaS}, we obtain the line
\begin{equation}\label{line:z}
z(t)=\tilde{\lambda}^{-\frac{k}{q-k}}\lambda(s)^{\frac{q}{q-k}},\; t\in \RR.
\end{equation}
Note that this line has range $(0,\tilde{\lambda})$. Further $\lim_{t\to -\infty}z(t) =0$ and $\lim_{t\to +\infty}z(t) =\tilde{\lambda}$ by \eqref{As:lambdaS}. Moreover, for each intersection between the line and the orbit, we obtain one and/or several times $t_0$. For each $t_0$ we define
\[
u_{\lambda(s_0)}(r) = 1 + (1-A(s_0))v(s),
\]
with $s_0=e^{t_0}$ and $s$ is given by \eqref{def:s:r}. Hence, in the case $q^*(k) <q < q_{JL}(k)$, as time increases, the line \eqref{line:z} intersects the orbit $(y(t),z(t))$ either a finite large number of times when $\lambda$ approaches $\tilde{\lambda}$, or an infinite number of times when $\lambda=\tilde{\lambda}$. In the case $q\geq q_{JL}(k)$, we see that the line \eqref{line:z} intersects the orbit $(y(t),z(t))$ only one time for each time $t_0$, which in turn implies that the problem $(P_{\lambda(s_0)})$ admits a unique solution $u_{\lambda(s_0)}$.  Now, if $\tilde{\lambda} <\lambda^*$ we see that the solution $u_{\lambda(s_0)}$ is a maximal solution by Theorem \ref{Main:3:Intro}, which is decreasing as $\lambda(s_0)$ increases. Since $\lambda(s_0) \to \tilde{\lambda}$ as $s_0\to +\infty$, we conclude that $u_{\lambda(s_0)}(0) \to A(\infty) >-\infty$ as $s_0\to +\infty$ since $\tilde{\lambda} <\lambda^*$, which is impossible by the change of variable \eqref{EF:var}-\eqref{y-z:def}. Therefore, $\tilde{\lambda} =\lambda^*$. This completes the proof.


\end{document}